\documentclass[oneside,american]{amsart}
\usepackage[T1]{fontenc}
\usepackage[latin9]{inputenc}
\usepackage{babel}
\usepackage{refstyle}
\usepackage{mathrsfs}
\usepackage{enumitem}
\usepackage{amstext}
\usepackage{amsthm}
\usepackage{amssymb}
\usepackage{mathtools}
\usepackage[unicode=true,pdfusetitle,
 bookmarks=true,bookmarksnumbered=false,bookmarksopen=false,
 breaklinks=false,pdfborder={0 0 1},backref=false,colorlinks=false]
 {hyperref}
\usepackage{todonotes}
\usepackage{pgfplots}
\pgfplotsset{compat=1.14} 

\makeatletter

\AtBeginDocument{\providecommand\thmref[1]{\ref{thm:#1}}}
\AtBeginDocument{\providecommand\eqref[1]{\ref{eq:#1}}}
\AtBeginDocument{\providecommand\propref[1]{\ref{prop:#1}}}
\RS@ifundefined{subsecref}
  {\newref{subsec}{name = \RSsectxt}}
  {}
\RS@ifundefined{thmref}
  {\def\RSthmtxt{theorem~}\newref{thm}{name = \RSthmtxt}}
  {}
\RS@ifundefined{lemref}
  {\def\RSlemtxt{lemma~}\newref{lem}{name = \RSlemtxt}}
  {}

\numberwithin{equation}{section}
\numberwithin{figure}{section}
\theoremstyle{plain}
\newtheorem{thm}{\protect\theoremname}[section]
  \theoremstyle{plain}
  \newtheorem{cor}[thm]{\protect\corollaryname}
  \theoremstyle{plain}
  \newtheorem{prop}[thm]{\protect\propositionname}
  \theoremstyle{plain}
  \newtheorem{lem}[thm]{\protect\lemmaname}
  \theoremstyle{remark}
  \newtheorem*{rem*}{\protect\remarkname}
  \theoremstyle{definition}

\newref{prop}{name=Proposition~}
\newref{thm}{name=Theorem~}
\newref{lem}{name=Lemma~}
\newref{cor}{name=Corollary~}
\newref{conj}{name=Conjecture~}
\newref{definition}{name=Definition~}
\newref{fig}{name=Figure~}

\makeatother

  \providecommand{\corollaryname}{Corollary}
  \providecommand{\lemmaname}{Lemma}
  \providecommand{\propositionname}{Proposition}
  \providecommand{\remarkname}{Remark}
\providecommand{\theoremname}{Theorem}

\def\R {\mathbb{R}}

\def\eps{\varepsilon}

\newcommand{\id}{\mathrm{Id}}

\DeclareMathOperator{\ind}{index}

\begin{document}

\title[Competition in periodic media, III]{Competition in periodic media: III \textendash{} Existence \& stability
of segregated periodic coexistence states}

\author{L\'{e}o Girardin$^{1}$ \and Alessandro Zilio$^{2}$}

\thanks{The research leading to these results has received funding from the
European Research Council under the European Union's Seventh Framework
Programme (FP/2007-2013) / ERC Grant Agreement n.321186 \textendash{}
ReaDi \textendash{} Reaction\textendash Diffusion Equations, Propagation
and Modelling held by Henri Berestycki. \\
A. Zilio was supported also by the ERC Advanced Grant 2013 n. 339958 \textendash{} COMPAT \textendash{} Complex Patterns for Strongly Interacting Dynamical Systems held by Susanna Terracini. \\
$^{1}$ Laboratoire Jacques-Louis Lions, CNRS UMR 7598, Universit\'{e}
Pierre et Marie Curie, 4 place Jussieu, 75252 Paris Cedex 05, France.\\
$^{2}$ Laboratoire Jacques-Louis Lions, CNRS UMR 7598, Universit\'{e}
Paris Diderot, B\^{a}timent Sophie Germain, 75205 Paris Cedex 13, France.}

\email{$^{1}$ girardin@ljll.math.upmc.fr}

\email{$^{2}$ azilio@math.univ-paris-diderot.fr}
\begin{abstract}
In this paper we consider a system of parabolic reaction-diffusion equations with strong competition and two related scalar reaction-diffusion equations. We show that in certain space periodic media with large periods, there exist periodic, non-constant, non-trivial, stable stationary states. We compare our results with already known results about the existence and nonexistence of such solutions. Finally, we provide ecological interpretations for these results.
\end{abstract}

\keywords{Competition\textendash diffusion system, periodic media, segregation,
stability.}

\subjclass[2000]{35B10, 35B35, 35B40, 35K57, 92D25.}
\maketitle

\section{Introduction}

We construct stable periodic sign-changing steady states in one-dimensional
spatially periodic media for the equation
\begin{equation}
\partial_{t}z-\partial_{xx}z=f\left(z,x\right)\label{eq:parabolic_semilinear_equation}
\end{equation}
and its quasi-linear counterpart
\begin{equation}
\partial_{t}\left(\sigma(z)z\right)-\partial_{xx}z=f\left(z,x\right),\label{eq:parabolic_quasilinear_equation}
\end{equation}
 where 
\[
f:\left(z,x\right)\mapsto\mu_{1}(x)\left(a_{1}-\frac{1}{\alpha}z\right)z^{+}-\frac{1}{d}\mu_{2}(x)\left(a_{2}+\frac{1}{d}z\right)z^{-}
\]
and the positive function $\sigma$ is
\[
\sigma:z\mapsto\mathbf{1}_{z>0}+\frac{1}{d}\mathbf{1}_{z<0}.
\]
Here $L$, $a_{1}$, $a_{2}$, $\alpha$ and $d$ are positive
constants, $\mu_{1},\mu_{2}\in L^{\infty}\left(\mathbb{R},\left(0,+\infty\right)\right)$
are positive $L$-periodic functions, $z^{+}=\max\left(z,0\right)$ and $z^{-}=-\min\left(z,0\right)$
(so that $z=z^{+}-z^{-}$).

We also construct stable periodic coexistence steady states for the following competition\textendash diffusion
system:
\begin{equation}
\begin{cases}
	\partial_{t}u_{1}-\partial_{xx}u_{1}=\mu_{1}(x)\left(a_{1}-u_{1}\right)u_{1} - k\omega(x)u_{1} u_{2}\\
	\partial_{t}u_{2}-d\partial_{xx}u_{2}=\mu_{2}(x)\left(a_{2}-u_{2}\right)u_{2} - \alpha k\omega(x)u_{1} u_{2}
\end{cases}\label{eq:parabolic_system}
\end{equation}
where $\omega\in L^{\infty}\left(\mathbb{R},\left(0,+\infty\right)\right)$ is positive and $L$-periodic (with a normalized mean value, say).

System (\ref{eq:parabolic_system}) belongs to the wider class of elliptic or parabolic systems of Lotka--Volterra type in the presence 
of strong competition, and (\ref{eq:parabolic_semilinear_equation}) and (\ref{eq:parabolic_quasilinear_equation}) 
are related to its singular \textit{strong competition limit} $k\to+\infty$. To our knowledge, the study of the 
strong competition limit appeared first in \cite{Dancer_Du_1994} as a way to model 
biological species that are fiercely competing for the same resource. The literature on this subject is very vast, 
varying from existence and uniqueness results \cite{Dancer_1991}, multiplicity results in presence 
of strong competition \cite{Dancer_Du_1994} and the rigorous proof of Gause's
competitive exclusion \cite{Kishimoto_Wein, Dancer_Hilhors} stating that in the homogeneous case, 
non-constant solutions are necessarily unstable (in convex domains).  We refer the interested reader 
to these contributions and the references therein.

More recently, the strong competition limit in periodic media was the object of investigation of two papers
\cite{Girardin_2016,Girardin_Nadin_2016} by the first author and
Nadin. According to \cite{Girardin_Nadin_2016}, (\ref{eq:parabolic_quasilinear_equation})
is the equation satisfied, in the strong competition limit, by the quantity 
$\alpha u_{1}-du_{2}$ with $\left(u_{1},u_{2}\right)$
solution of (\ref{eq:parabolic_system}). Notice that, by normalizing $\left(u_1,u_2\right)$,
we can assume without loss of generality $a_1=a_2=1$. This is assumed indeed from now on. Notice
also that, although all results of \cite{Girardin_2016,Girardin_Nadin_2016} are stated for $\omega=1$, they are 
readily extended to the case of non-constant $\omega$.

Steady states of (\ref{eq:parabolic_semilinear_equation}) and of
(\ref{eq:parabolic_quasilinear_equation}) satisfy the same elliptic semilinear equation:
\begin{equation}
-z''(x)=\mu_{1}(x)\left(1-\frac{1}{\alpha}z(x)\right)z^{+}(x)-\frac{1}{d}\mu_{2}(x)\left(1+\frac{1}{d}z(x)\right)z^{-}(x).\label{eq:elliptic_equation}
\end{equation}
However, due to the different time dependencies, (\ref{eq:parabolic_semilinear_equation})
and (\ref{eq:parabolic_quasilinear_equation}) involve in general
different notions of stability and therefore different eigenproblems.
Before going any further, let us precise this important point. 

\subsection{Notions of stability}
For any functional space $X$, $X_{L\textup{-per}}$
denotes the set of $L$-periodic functions whose restriction to 
any interval of length $L$ are elements of $X$. Accordingly,
for any second order monotone elliptic operator $\mathscr{L}$, $\lambda_{1,L\textup{-per}}\left(-\mathscr{L}\right)$
denotes the periodic principal eigenvalue of $\mathscr{L}$ given
by the Krein\textendash Rutman theorem. Recall that if $\left(u_{1},u_{2}\right)$
is a solution of (\ref{eq:parabolic_system}), then the system satisfied
by $\left(u_{1},1-u_{2}\right)$ is a monotone system, whence
its linearization admits indeed a periodic principal eigenvalue (details
can be found in \cite{Girardin_2016}). 

Hereafter, a solution $z\in H^{2}_{L\textup{-per}}\left(\R\right)$ of (\ref{eq:elliptic_equation})
such that the $L$-periodic function
\[
f_{1}\left[z\right]:x\mapsto\partial_{1}f\left(z(x),x\right),
\]
is well-defined (at least weakly) is referred to as \textit{linearly
stable in the sense of (\ref{eq:parabolic_semilinear_equation})}
if 
\[
\lambda_{1,\textup{$L$-per}}\left(-\frac{\mathrm{d}^{2}}{\mathrm{d}x^{2}}-f_{1}\left[z\right]\right)>0
\]
and as \textit{linearly stable in the sense of (\ref{eq:parabolic_quasilinear_equation})}
if 
\[
\lambda_{1,\textup{$L$-per}}\left(-\hat{\sigma}(z)\frac{\mathrm{d}^{2}}{\mathrm{d}x^{2}}-\hat{\sigma}(z)f_{1}\left[z\right]\right)>0,
\]
with 
\[
\hat{\sigma}:z\mapsto\mathbf{1}_{z\geq0}+d\mathbf{1}_{z<0}.
\]
The constant solutions of (\ref{eq:elliptic_equation}) are
$\alpha$, $-d$ and $0$. It is easily verified that $\alpha$
and $-d$ are linearly stable in both senses whereas $0$ is
linearly unstable (namely, not linearly stable) in both senses. 

The definition of linear stability in the sense of (\ref{eq:parabolic_quasilinear_equation})
can be formally understood by plugging perturbations of the form $\textup{e}^{-\lambda t}\varphi(x)$,
with $\varphi$ $L$-periodic, into the equation (\ref{eq:parabolic_quasilinear_equation})
linearized at an almost everywhere nonzero steady state $z$. Indeed,
such a perturbation solves the linear equation if and only if 
\[
-\lambda\sigma(z)\varphi-\varphi''=f_{1}\left[z\right]\varphi,
\]
that is, due to the almost everywhere equality $\sigma\left(z(x)\right)\hat{\sigma}\left(z(x)\right)=1$,
if and only if 
\[
-\hat{\sigma}(z)\varphi''-\hat{\sigma}(z)f_{1}\left[z\right]\varphi=\lambda\varphi.
\]

Similarly, a steady state solution $\left(u_{1},u_{2}\right)$ of
(\ref{eq:parabolic_system}) is a solution of 
\begin{equation}
\left\{ \begin{matrix}-u_{1}''(x)=\mu_{1}(x)\left(1-u_{1}(x)\right)u_{1}(x)-k\omega(x)u_{1}(x)u_{2}(x)\\
-d u_{2}''(x)=\mu_{2}(x)\left(1-u_{2}(x)\right)u_{2}(x)-\alpha k\omega(x)u_{1}(x)u_{2}(x)
\end{matrix}\right.\label{eq:elliptic_system}
\end{equation}
and is referred to as \textit{linearly stable}
if 
\[
\lambda_{1,\textup{$L$-per}}\left(-\left(\begin{matrix}\frac{\mathrm{d}^{2}}{\mathrm{d}x^{2}}+\mu_{1}\left(1-2u_{1}\right)-k\omega u_{2} & k\omega u_{1}\\
\alpha k\omega u_{2} & d\frac{\mathrm{d}^{2}}{\mathrm{d}x^{2}}+\mu_{2}\left(1-2u_{2}\right)-\alpha k\omega u_{1}
\end{matrix}\right)\right)>0.
\]
The steady states $\left(1,0\right)$ and $\left(0,1\right)$
are linearly stable whereas $\left(0,0\right)$ is linearly unstable. 

By analogy with the spatially homogeneous setting and in view of the
stability of the constant solutions, (\ref{eq:parabolic_semilinear_equation}),
(\ref{eq:parabolic_quasilinear_equation}) and (\ref{eq:parabolic_system})
are sometimes referred to as \textit{bistable}. However our main contribution is to prove that this terminology can be misleading:
because of the spatial heterogeneity, a third stable state can very
well exist.

Let us point out that the previous two parts of the series 
\textquotedblleft \textit{Competition in periodic media}\textquotedblright \cite{Girardin_2016,Girardin_Nadin_2016} only
used the notion of stability in the sense of the system (\ref{eq:parabolic_system}).
This explains why the two notions of stability for the segregated equation (\ref{eq:elliptic_equation}) are only introduced now.

\subsection{Main results}

Let $\left(r_0,r_1,r_2\right)\in\left(0,1\right)^3$ such that $2r_0+2r_1+2r_2=1$.
Let $\left(M_1,M_2\right)\in\left(0,+\infty\right)^2$ and define two $1$-periodic functions $\mu_1^\star$ and $\mu_2^\star$ by
\[
\left(\mu_1^\star\right)_{|\left[0,1\right]}=M_1\mathbf{1}_{\left[0,r_1\right]}+M_1\mathbf{1}_{\left[r_1+2r_0+2r_2,1\right]}
\]
\[
\left(\mu_2^\star\right)_{|\left[0,1\right]}=M_2\mathbf{1}_{\left[r_1+r_0,r_1+r_0+2r_2\right]}
\]
and, for all $L>0$, 
\[
\left(\mu_1^L,\mu_2^L\right):x\mapsto\left(\mu_1^\star,\mu_2^\star\right)\left(\frac{x}{L}\right).
\]

Our first main result is concerned with the equation (\ref{eq:elliptic_equation}). 

\begin{thm}
\label{thm:For_the_segregated_bistable_equation} There  exists $\underline{L}>0$ such
that, for all $L>\underline{L}$, (\ref{eq:elliptic_equation}) with $\left(\mu_1,\mu_2\right)=\left(\mu_1^L,\mu_2^L\right)$
or with $\left(\mu_1,\mu_2\right)=\left(\mu_1^L+\mu_2^L,\mu_1^L+\mu_2^L\right)$
admits a linearly stable in both senses, sign-changing, $L$-periodic solution.

Furthermore, for all $L>\underline{L}$, there exist a neighborhood $U_L$ of $\left(\mu_{1}^L,\mu_{2}^L\right)$ in the
topology of $\left(L_{L\textup{-per}}^{\infty}\right)^{2}$ and a neighborhood $V_L$ of $\mu_{1}^L+\mu_{2}^L$ in the
topology of $\left(L_{L\textup{-per}}^{\infty}\right)$ such that,
for all $\left(\mu_{1},\mu_{2}\right)\in U_L$ and all $\mu\in V_L$, (\ref{eq:elliptic_equation}) with 
$\left(\mu_1,\mu_2\right)$ or $\left(\mu,\mu\right)$
admits a linearly stable in both senses, sign-changing, $L$-periodic solution.
\end{thm}

This first result will be proved by explicit construction of $v$ and non-trivial application of
the implicit function theorem. 

In biological terms, the growth rate $\mu_{1}^L+\mu_{2}^L$ corresponds
to a periodic environment where large favorable areas are separated by large neutral areas.
A neutral area could be, say,
in a woodland inhabited by herbivorous animals looking for glades, 
an area densely covered by trees where predators live and hide and where linear
death rates roughly equal linear birth rates and no intraspecific competition occurs. The associated stable
steady state describes the situation where one competitor settles in the evenly numbered
favorable areas whereas the other settles in the oddly numbered ones. This particular form
is illustrated by \figref{mu_i_and_v}.

Let us point out that well-known density results yield immediately the following corollary.
\begin{cor}
\label{cor:With_smooth_coefficients}For all $L>\underline{L}$, there exists $\left(\mu_{1},\mu_{2}\right)\in\left(\mathscr{C}_{\textup{$L$-per}}^{\infty}\left(\mathbb{R},\left(0,+\infty\right)\right)\right)^{2}$
such that (\ref{eq:elliptic_equation}) admits a linearly stable in both
senses, sign-changing, $L$-periodic solution.
\end{cor}

Our second main result is concerned with the system (\ref{eq:elliptic_system}) and states that 
the existence of stable steady states for the segregated equation implies the existence 
of stable steady states for the strongly competitive system. It will be proved as a consequence of 
\thmref{For_the_segregated_bistable_equation} and of degree theory. 

\begin{thm}
\label{thm:For_the_strongly_competing_system} For all $L>\underline{L}$, there exist $k^\star>0$ and 
$\left(\mu_{1},\mu_{2}\right)\in\left(\mathscr{C}_{\textup{$L$-per}}^{\infty}\left(\mathbb{R},\left(0,+\infty\right)\right)\right)^{2}$
such that, for all $k>k^\star$, (\ref{eq:elliptic_system}) admits a linearly stable, component-wise positive, $L$-periodic
solution.
\end{thm}

\subsection{Discussion and comparison with known results}

\thmref{For_the_segregated_bistable_equation} and \thmref{For_the_strongly_competing_system}
complement interestingly a result of the first author \cite[Theorem 1.2]{Girardin_2016}
stating that, provided $L$ is sufficiently small, that is
\[
L\in\left(0,\pi\left(\left(\max_{\left[0,L\right]}\mu_{1}\right)^{-\frac{1}{2}}+\sqrt{d}\left(\max_{\left[0,L\right]}\mu_{2}\right)^{-\frac{1}{2}}\right)\right),
\]
 and provided $k$ is large enough, all $L$-periodic coexistence states are
unstable and vanish as $k\to+\infty$. 

\thmref{For_the_segregated_bistable_equation} is also directly related
to a result due to Ding, Hamel and Zhao \cite[Theorem 1.5]{Ding_Hamel_Zhao}
which shows in particular that the regular bistable equation
\[
\partial_t z-\partial_{xx} z=g_L(x,z),
\]
with $g_L:\left(z,x\right)\mapsto g\left(z,\frac{x}{L}\right)$, $g$ 
$1$-periodic with respect to $x$ and independent of $L$, $0$ and $1$ linearly stable steady states (in the standard
sense) and $\theta\in\mathscr{C}_{1-\textup{per}}\left(\mathbb{R},\left(0,1\right)\right)$
intermediate zero of $g$, admits bistable pulsating fronts connecting
$0$ and $1$ provided $L$ is large enough and the nonlinearity
$g$ satisfies 
\[
\min_{x\in\left[0,L\right]}\int_{0}^{1}g(x,z)\mathrm{d}z>0\textup{ and }
\min_{x\in\left[0,L\right]}\frac{\partial g}{\partial z}\left(x,\theta(x)\right)>0.
\]
Their proof is based on a very important result by Fang and Zhao \cite{Fang_Zhao_2011}
stating in a general setting that bistable pulsating fronts exist
if all intermediate periodic steady states are unstable and invadable.
Therefore the proof of Ding\textendash Hamel\textendash Zhao basically
shows that the above conditions imply the nonexistence of
stable periodic steady states. Importantly, 
\begin{itemize}
\item on one hand, the family of scaled functions $\left(f_L\right)_{L>\underline{L}}$ 
in \thmref{For_the_segregated_bistable_equation} satisfies 
\[\min_{x\in\left[0,L\right]}\int_{-d}^{\alpha}f_L(x,z)\mathrm{d}z=0\quad\textup{ for all }L>\underline{L}
\]
(recalling that here the two constant stable states are $-d$ and $\alpha$ instead of $0$ and $1$);
\item on the other hand, any family of regularized and positive functions obtained from \corref{With_smooth_coefficients} satisfies
indeed the above two positivity conditions, but by the result of Ding\textendash Hamel\textendash Zhao cannot be of the prescribed scaled form as $L$ varies (in other words, the neighborhoods
$U_L$ and $V_L$ obtained with the implicit function theorem are not uniform with respect to $L$ and shrink as $L\to+\infty$).
\end{itemize}

We point out that a recent paper by Zlat\v{o}s \cite{Zlatos_2017} constructed an example of periodic bistable
nonlinearity admitting no pulsating front. His result is very related to ours but remains qualitatively
different: we focus on stable intermediate steady states whereas Zlat\v{o}s focuses on nonexistence of
transition fronts. Furthermore, our construction has a very simple ecological interpretation and is
valid for all large periods, whereas the construction of Zlat\v{o}s requires a very precise period. In this 
regard, our paper is an interesting complement.

\thmref{For_the_segregated_bistable_equation} is also related to
a family of results stating, loosely speaking, that the geometry of
a homogeneous domain with boundary can block bistable propagation.
See for instance Berestycki\textendash Bouhours\textendash Chapuisat
\cite{Berestycki_Bouhours_Chapuisat} and references therein.
Although we do not prove that our periodic stable steady state is able to block the propagation of
a constant stable steady state, its mere existence makes it impossible to apply the theory of
Fang\textendash Zhao \cite{Fang_Zhao_2011} so that the existence of pulsating fronts remains
unclear. We might study in a future work whether blocking occurs or not in our case. 

Ecologically speaking, \thmref{For_the_strongly_competing_system}
shows that \textit{strong interspecific competition} and \textit{heterogeneity
of the habitat} can lead together to \textit{spatial segregation}
and therefore to \textit{speciation} and \textit{increased biodiversity}.
Having this interpretation in mind, we notice that the strength of
the competition is crucial: indeed, in the weak competition case,
Dockery\textendash Hutson\textendash Mischaikow\textendash Pernarowski
\cite{Dockery_1998} showed on the contrary that heterogeneity leads
to extinction of all competitors but the one with the lowest diffusion
rate. Ecologically, strong competition occurs for instance when resources
are rare. Mathematically, it is known to lead indeed to spatial segregation,
or in other words pattern formation, in homogeneous domains with appropriate
boundary conditions or initial conditions (see for instance \cite{Conti_Terracin,Crooks_Dancer_,Dancer_Hilhors}
and references therein). As such, our result can be seen as a contribution
to the overarching research program on pattern formation in strongly
competing systems and as one of the first results in spatially heterogeneous
domains.

It is worthy to recall that by a result of Berestycki\textendash Hamel\textendash Rossi
\cite[Proposition 6.6]{Berestycki_Hamel_Rossi}, the periodic principal
eigenvalue of a self-adjoint periodic scalar elliptic operator coincides
with the decreasing limit as $R\to+\infty$ of its Dirichlet principal
eigenvalue in the ball $\left(-R,R\right)$. Consequently, if the
domain of a linearly stable in both senses, periodic, sign-changing
steady state solution $z$ of (\ref{eq:elliptic_equation})
is restricted to a periodicity cell $\left(y,y+L\right)$ with $y$
chosen so that $z\left(y\right)=0$, then we obtain a steady state
for the corresponding Dirichlet problem which is linearly stable in
the following senses:
\[
\lambda_{1,\textup{Dir}}\left(-\frac{\mathrm{d}^{2}}{\mathrm{d}x^{2}}-f_{1}\left[z\right],\left(y,y+L\right)\right)>0,
\]
\[
\lambda_{1,\textup{Dir}}\left(-\hat{\sigma}(z)\frac{\mathrm{d}^{2}}{\mathrm{d}x^{2}}-\hat{\sigma}(z)f_{1}\left[z\right],\left(y,y+L\right)\right)>0.
\]

\subsection{What about more general bistable equations?} The particular shape of function $f$ in (\ref{eq:elliptic_equation}) 
is due to the underlying ecological model. With very few modifications,
\thmref{For_the_segregated_bistable_equation} can be extended to more general
bistable equations in periodic media, like for instance the familiar Allen--Cahn equation 
\[
	\partial_t z - \partial_{xx} z = \mu_L(x) (1-z^2)z.
\]

\subsection{Structure of the paper} In Section 2, we prove \thmref{For_the_segregated_bistable_equation}, focusing first on
the construction of $v$ and then using the implicit function theorem to obtain the 
open neighborhood $U$. In Section 3, we prove \thmref{For_the_strongly_competing_system} thanks to 
\thmref{For_the_segregated_bistable_equation} and topological arguments.

\section{The segregated bistable equation}
Our goal in this section is to prove that (\ref{eq:elliptic_equation}) admits sign-changing solutions that are also stable in 
the sense of (\ref{eq:parabolic_semilinear_equation}) and (\ref{eq:parabolic_quasilinear_equation}). 

Before going any further, we observe the following: replacing 
$\left(\frac{\mu_1}{\alpha},\frac{\mu_2}{d^2}\right)$ by $\left(\mu_1,\mu_2\right)$,
(\ref{eq:elliptic_equation}) reads
\begin{equation}
	-z''=\mu_{1}\left(\alpha -z\right)z^{+}-\mu_{2}\left(d + z\right)z^{-}.\label{eq:normalized_elliptic_equation}
\end{equation}
Hence up to end of this section we have in mind the above more compact form. The piecewise-constant functions 
$\mu_1^\star$ and $\mu_2^\star$ defined in the introduction are accordingly modified, with 
$\left(\frac{M_1}{\alpha},\frac{M_2}{d^2}\right)$ replaced by $\left(M_1,M_2\right)$.

In order to construct a sing-changing, periodic and stable solution to (\ref{eq:normalized_elliptic_equation}), we need a preliminary result concerning its linearization.

\subsection{Linearization near a non-constant stationary solution}

Since the right hand side of (\ref{eq:normalized_elliptic_equation}) is only Lipschitz continuous at $z=0$, we need some caution in 
order to properly introduce the linearization of the equation around a sign-changing steady state. Many authors have already 
addressed similar issues (see, for instance, \cite[Section 4.1]{Dancer_Hilhors}). Since we could not find the precise statement 
that we needed, we decided to present a complete proof. We wish to point out that the result can be adapted to more general 
equations (for instance bounded domains with Neumann boundary conditions).

For all $\left(\mu_{1},\mu_{2},z\right)\in\left(L_{L\textup{-per}}^{\infty}\right)^{2}\times H_{L\textup{-per}}^{2}$, we define 
\[
	\mathscr{F}: \left(L_{L\textup{-per}}^{\infty}\right)^{2}\times H_{L\textup{-per}}^{2} \to L_{L\textup{-per}}^{2}
\]
such that, for all test functions $\varphi \in H_{L\textup{-per}}^{2}$,
\begin{equation}
\label{eq:nonlinear functional}
	\left<\mathscr{F}(\mu_1, \mu_2, z),\varphi\right> = \int_{0}^{L}z'\varphi'-\int_{0}^{L}\left(\mu_{1}\left(\alpha-z\right)z^{+}-\mu_{2}\left(d+z\right)z^{-}\right)\varphi.
\end{equation}

We recall that, by Sobolev embedding, the inclusion $H^{2}_{L\textup{-per}} \hookrightarrow \mathscr{C}^{1,\frac{1}{2}}_{L\textup{-per}}$ 
holds true.

\begin{lem}
\label{lem:C1_solution_operator} Let $O\subset H_{L\textup{-per}}^{2}$ be an open set in the topology of $H_{L\textup{-per}}^{2}$ such that for all $z\in O$, the closed set $z^{-1}\left(\left\{ 0\right\} \right)$ has zero Lebesgue measure.

Then $\mathscr{F}\in\mathscr{C}^{1}\left(\left(L_{L\textup{-per}}^{\infty}\right)^{2}\times O, L_{L\textup{-per}}^{2} \right)$.

For any $\left(\mu_{1},\mu_{2},z\right)\in\left(L_{L\textup{-per}}^{\infty}\right)^{2}\times O$
and any $\left(\eta_{1},\eta_{2},w\right)\in\left(L_{L\textup{-per}}^{\infty}\right)^{2}\times H_{L\textup{-per}}^{2}$, the differential
$\mathrm{d}\mathscr{F}\left[\mu_{1},\mu_{2},z\right]$ evaluated at $\left(\eta_{1},\eta_{2},w\right)$ is
\begin{multline*}
\varphi\mapsto\int_{0}^{L}w'\varphi'-\int_{0}^{L}\left(\eta_{1}\left(\alpha-z\right)z^{+}-\eta_{2}\left(d+z\right)z^{-}\right)\varphi\\
-\int_{0}^{L}\left(\mu_{1}\left(\alpha-2z\right)\mathbf{1}_{z>0}+\mu_{2}\left(d+2z\right)\mathbf{1}_{z<0}\right)w\varphi.
\end{multline*}
\end{lem}

\begin{rem*}
Some assumptions on the open set $O$ are necessary. In general, the G\^{a}teaux differential of $\mathscr{F}$ at $\left(\mu_{1},\mu_{2},z\right)$ in the direction $\left(\eta_{1},\eta_{2},w\right)$ fails to be linear with respect to $\left(\eta_{1},\eta_{2},w\right)$.
More precisely, it is the sum of the linear functional above and of 
\[
\varphi\mapsto-\int_{0}^{L}\left(\mu_{1}\alpha w^{+}-\mu_{2}d w^{-}\right)\mathbf{1}_{z=0}\varphi,
\]
which is non-linear with respect to $w$. We can prove this by partitioning $\mathbb{R} = \{z > 0\} \cup \{z = 0\} \cup \{z<0\}$.
\end{rem*}

\begin{proof}
The linear mapping appearing in the statement above is readily continuous. Thus we only need to show that it is indeed the 
G\^{a}teaux differential. 

Fix $\left(\mu_{1},\mu_{2},z\right)\in\left(L_{L\textup{-per}}^{\infty}\right)^{2}\times O$
and $\left(\eta_{1},\eta_{2},w\right)\in\left(L_{L\textup{-per}}^{\infty}\right)^{2}\times H_{L\textup{-per}}^{2}$. 
For all $t>0$ and all $\varphi\in H_{L\textup{-per}}^{2}$, 
\begin{multline*}
\frac{1}{t}\left(\mathscr{F}\left[\left(\mu_{1},\mu_{2},z\right)+t\left(\eta_{1},\eta_{2},w\right)\right]-\mathscr{F}\left[\left(\mu_{1},\mu_{2},z\right)\right]\right)\left(\varphi\right) = \\
\int_{0}^{L}w'\varphi'
 -\frac{1}{t}\int_{0}^{L}\left(\left(\mu_{1}+t\eta_{1}\right)\left(\alpha-(z+tw)\right)(z+tw)^+ -\mu_{1}(\alpha-z)v^+\right)\varphi \\
 +\frac{1}{t}\int_{0}^{L}\left(\left(\mu_{2}+t\eta_{2}\right)\left(d+(z+tw)\right)(z+tw)^- -\mu_{2}(d-z)z^-\right)\varphi.
\end{multline*}
The first term in the right hand side does not depend on $t$. We only need to consider the second one, as the third one can be 
dealt with in a similar way. Rearranging the terms, we find 
\begin{multline*}
	\frac{1}{t}\int_{0}^{L}\left(\left(\mu_{1}+t\eta_{1}\right)\left(\alpha-(z+tw)\right)(z+tw)^+ - \mu_{1}(\alpha-z)v^+ \right)\varphi\\
	=\int_{0}^{L}\eta_{1}\left(\alpha-(z+tw)\right)(z+tw)^+ \varphi\\
    +\int_{0}^{L}\mu_1 \frac{\left(\alpha-(z+tw)\right)(z+tw)^+ - (\alpha-z)v^+}{t} \varphi.
\end{multline*}
The dominated convergence theorem yields
\[
\int_{0}^{L}\eta_{1}\left(\alpha-(z+tw)\right)(z+tw)^+ \varphi\to\int_{0}^{L}\eta_{1}\left(\alpha-z \right)v^+ \varphi \qquad \textup{as }t\to0.
\]

Rearranging the last term of the preceding equality, we find
\begin{multline*}
\int_{0}^{L}\mu_1 \left(\frac{\left(\alpha-z-tw\right)\left(z+tw\right)^{+}-\left(\alpha-z\right)z^{+}}{t}\right) \varphi\\
	=\int_{0}^{L}\mu_1 \left(\frac{\left(z+tw\right)^{+}-z^{+}}{t}\right)\left(\alpha-z\right) \varphi-\int_0^L \mu_1  w \alpha \left(z+tw\right)^{+} \varphi.
\end{multline*}
 By dominated convergence, 
\[
\lim_{t\to0}\int_0^L \mu_1 w \alpha \left(z+tw\right)^{+} \varphi = \int_0^L \mu_1 w \alpha z^{+} \varphi.
\]
Since by assumption $z^{-1}(\{0\})$ has zero Lebesgue measure and the map $\zeta\mapsto\zeta^+$ is smooth away from $0$, the dominated convergence theorem yields once again
\[
\lim_{t\to0}\int_{0}^{L}\mu_1 \left(\frac{\left(z+tw\right)^{+}-z^{+}}{t}\right)\left(\alpha-z\right) \varphi=\int_{0}^{L}\mu_{1} w \mathbf{1}_{z>0}\left(\alpha - z\right) \varphi.
\]

This concludes the proof. 
\end{proof}

\subsection{Construction of the solution} We now proceed by constructing the solution of (\ref{eq:normalized_elliptic_equation}). To do so, we first consider the equation with piecewise-constant coefficients. In this case, solutions can be constructed by gluing together different profiles. The implicit function theorem then leads to an open neighborhood of valid coefficients near this
piecewise-constant pair.

\subsubsection{Piecewise-constant coefficients} In the following result we collect some properties of the solutions of the logistic equation with non-zero Dirichlet conditions. These properties are well known and straightforward consequences of the comparison principle. For this reason, we do not present here a fully detailed proof.

\begin{lem}
\label{lem:Dirichlet_solutions} For all $A>0$, $M>0$, $\nu \in \left[\frac{1}{2},1\right)$ and $R > 0$ there
exists a unique positive solution $w_{A, M,\nu, R}\in\mathscr{C}^{2}\left(\left[-R,R\right]\right)$ of 
\[
	\begin{cases}
	-w''= M \left(A-w\right)w &\textup{in }\left(-R,R\right)\\
	w\left(\pm R\right)= \nu A.
	\end{cases}
\]
The function $w_{A, M,\nu, R}$ is even and satisfies 
\[
	\nu A < w_{A, M,\nu, R}(x) < A \qquad \textup{for all $x\in\left(-R,R\right)$}.
\]

Furthermore, let 
\[
	\Phi : (A,M, \nu, R) \mapsto w'_{A,M,\nu,R}\left(-R\right).
\]
The following properties hold true.
\begin{enumerate}
\item $\Phi$ is positive and continuous;
\item it holds
\[
	\lim_{R\to 0^+} \Phi(A,M,\nu,R) = 0;
\]
\item there exists $\gamma_{A,M,\nu}\in\left(0,+\infty\right)$ such that
\[
	\gamma_{A,M,\nu} = \lim_{R\to+\infty} \Phi\left(A,M,\nu,R\right).
\]
Moreover, $(A,M,\nu) \mapsto \gamma_{A,M,\nu}$ is continuous with respect to $A$, $M$ and $\nu$, increasing with respect to
$A$ and $M$ and decreasing with respect to $\nu$. In particular 
$0 = \lim_{\nu \to 1} \gamma_{A,M,\nu} < \gamma_{A,M,\nu} < \gamma_{A,M,\frac{1}{2}}$;
\item the function $R \mapsto \Phi(A, M,\nu, R)$ is an increasing homeomorphism from $\left(0,+\infty\right)$ onto 
$\left(0,\gamma_{A,M,\nu}\right)$;
\item the function $\nu \mapsto \Phi(A, M,\nu, R)$ is a decreasing homeomorphism from $\left[\frac{1}{2},1\right)$ onto 
$\left(0,\Phi(A,M,\frac{1}{2},R)\right]$.
\end{enumerate}
\end{lem}

We point out that the upper limit $\gamma_{A,M,\nu}$ can actually be determined explicitly.

\begin{proof}
We perform the following change of variables
\[
	w(x) = A W_{\rho,\nu} \left(\sqrt{AM} x\right) \qquad \textup{and} \qquad \rho = \sqrt{AM} R.
\]
Here the function $W_{\rho,\nu}$ is a solution to the scaled equation
\begin{equation}\label{eq:scaled_log}
	\begin{cases}
	-W''= \left(1- W\right)W &\textup{in }\left(-\rho,\rho\right)\\
	W\left(\pm \rho\right)=\nu.
	\end{cases}
\end{equation}
We can rephrase all the statements of the result in terms of the dependence of $W_{\rho,\nu}$ on $\rho$ and $\nu$. Here we consider only the dependence on $\rho$. The same arguments can be adapted to show the corresponding results in terms of $\nu$.

For any value of $\rho > 0$ and $\nu \in [\frac{1}{2},1)$, the previous equation admits a unique, positive solution which is even and is such that $\nu < W(x) < 1$ for all $x \in (-\rho, \rho)$. This follows by standard arguments. We just observe that the functions $x \mapsto \nu \cos(\gamma x) / \cos(\gamma \rho)$ are sub-solutions of (\ref{eq:scaled_log}) for $\gamma$ small enough, while the constant $1$ is always a super-solution.

Notice that, for all $\kappa>1$:
\[
-\left(\kappa W_{\rho,\nu} \right)''=\left(1-W_{\rho,\nu}\right) \kappa W_{\rho,\nu} \geq \left(1-\kappa W_{\rho,\nu}\right) \kappa W_{\rho,\nu}\textup{ in }\left(-\rho,\rho\right).
\]
For all $\rho' > \rho > 0$, the following quantity is well-defined:
\[
\kappa^{\star}=\inf\left\{ \kappa>1\ |\ \kappa W_{\rho',\nu}\geq W_{\rho,\nu}\textup{ in }\left(-\rho,\rho\right)\right\} .
\]
 Assuming by contradiction that $\kappa^{\star}>1$ and applying the
strong maximum principle, we get a contradiction. Hence the family
$\left(W_{\rho,\nu}\right)_{\rho>0}$ is non-decreasing, and once
more by the strong maximum principle, it is in fact increasing. 

It follows that the function $\rho \mapsto \max_{[-\rho, \rho]} W_{\rho,\nu}(x)$ is increasing with limit $1$ as $\rho\to+\infty$. 
By classical elliptic estimates (see Gilbarg\textendash Trudinger \cite{Gilbarg_Trudin}) 
the family converges locally uniformly to a bounded and positive solution of (\ref{eq:scaled_log}) 
defined on the whole line $\R$. Hence, as $\rho \to +\infty$, we find that $W_{\rho,\nu} \to 1$ locally in $\mathscr{C}^{2}$.

We now consider the shifted family of functions
\[
 \overline{W}_{\rho,\nu}(x)= W_{\rho,\nu}\left(x-\rho\right) \qquad \textup{for} \quad x \in [0,2\rho].
\]
The family $\rho\mapsto\overline{W}_{\rho,\nu}$ is increasing. In particular, by the Hopf lemma,
\[
	\rho \mapsto \overline{W}_{\rho,\nu}'(0)
\]
is increasing as well. Once again, classical elliptic estimates show that, as $\rho \to +\infty$, the family $\overline{W}_{\rho,\nu}$ converges locally uniformly to the unique solution $\overline{W}$ of
\begin{equation}\label{eq:limit_R_infty_z_A,R}
\begin{cases}
	-\overline{W}''=\left(1-\overline{W}\right) \overline{W} & \textup{in }\left(0,+\infty\right)\\
	\overline{W}\left(0\right)=\nu\\
	\nu < \overline{W} < 1 & \textup{in }\left(0,+\infty\right)
\end{cases}
\end{equation}
(see Du\textendash Lin \cite[Proposition 4.1]{Du_Lin_2010,Du_Lin_2010_er}). Thus, the limit as $\rho \to +\infty$ of 
$\overline{W}_{\rho,\nu}'(-\rho)$ is finite and positive. We can figure out its value by testing (\ref{eq:limit_R_infty_z_A,R}) 
against $\overline{W}'$. This yields the identity
\[
	\lim_{\rho\to+\infty}\overline{W}'_{\rho,\nu}(-\rho) = \sqrt{\frac13 + \nu^2 \left(\frac23 \nu - 1\right)}.
\]
Observe that the limit is always positive and bounded. 

We conclude by observing that the continuity of $\overline{W}'_{\rho,\nu}$ with respect to $\rho$ 
is a classical consequence of the uniqueness of $\overline{W}_{\rho,\nu}$ and of compactness arguments.
\end{proof}

From the previous result we deduce a property which is crucial for our construction. For sake of brevity, from now on we will simply write
\[
	\Phi_1(\nu,L) = \Phi(\alpha,M_1,\nu,r_1 L),
\]
\[
	\Phi_2(\nu,L) = \Phi(d,M_2,\nu,r_2 L),
\]
(recalling that $M_1>0$, $M_2>0$, $r_1>0$ and $r_2>0$ were fixed in the introduction).

We can finally construct the periodic stable solutions of (\ref{eq:normalized_elliptic_equation}) with the piecewise-constant coefficients. 

\begin{prop}
\label{prop:Existence_for_the_scalar_eq_in_L_infty} 
There exists $\underline{L}> 0$ such that,
for any $L > \underline{L}$, (\ref{eq:normalized_elliptic_equation}) with either $\left(\mu_1,\mu_2\right)=\left(\mu_1^L,\mu_2^L\right)$
or with $\left(\mu_1,\mu_2\right)=\left(\mu_1^L+\mu_2^L,\mu_1^L+\mu_2^L\right)$ admits a nonzero sign-changing solution $v \in H_{L\textup{-per}}^{2}$ satisfying, for all $L$-periodic test functions 
$\varphi\in H_{L\textup{-per}}^{1}$,
\[
\int_{0}^{L}v' \varphi'=\int_{0}^{L}\left(\mu_{1}\left(\alpha-v\right)v^{+}-\mu_{2}\left(d+v\right) v^{-}\right)\varphi.
\]

Furthermore, $v$ is linearly stable in the sense of (\ref{eq:parabolic_semilinear_equation}) and (\ref{eq:parabolic_quasilinear_equation}). 
\end{prop}

\begin{figure}\label{fig:mu_i_and_v}
\begin{tikzpicture}
  \begin{axis}[
  	width = 12cm,
    height = 7cm,
    clip=false,
    axis x line = center,
    hide y axis,
    ticks=none,
    axis on top,
    samples     = 50,
    domain      = 0:5,
    xmin = 0, xmax = 5,
    ymin = -3, ymax = 3,
  ]
  
  \addplot[draw = none, fill = red!25] coordinates {(0,0) (1,0) (1,3) (0,3) (0,0)} \closedcycle;
  
  \addplot[draw = none, fill = blue!25] coordinates {(5,0) (4,0) (4,-3) (5,-3) (5,0)} \closedcycle;
  
  \addplot[dashed, black] coordinates {(0,2.85) (1,2.85)};
  \addplot[dashed, black] coordinates {(0,1.425) (1,1.425)};

  \addplot[dashed, black] coordinates {(4,-2.5) (5,-2.5)};
  \addplot[dashed, black] coordinates {(4,-1.25) (5,-1.25)};
  
    \pgfmathsetmacro{\zero}{2.16}
    
    \pgfmathsetmacro{\mone}{-2.93/\zero}
    \pgfmathsetmacro{\alfaone}{\mone / 2 *(1 - 2*\zero)}
    \pgfmathsetmacro{\betaone}{\mone / 2 }

    \addplot[black!30, thick, domain = 0:1]{ \alfaone + \betaone *x*x };
    \addplot[black!30, thick, domain = 1:4]{ \mone*(x-\zero) };   
    \addplot[black!30, thick, domain = 4:5]{ -2.5+0*x };
    
    \pgfmathsetmacro{\zero}{2.83}

    \pgfmathsetmacro{\mone}{-2.93/\zero}
    \pgfmathsetmacro{\alfaone}{\mone / 2 *(1 - 2*\zero)}
    \pgfmathsetmacro{\betaone}{\mone / 2 }

    \addplot[black!30, thick, domain = 0:1]{ \alfaone + \betaone *x*x };
    \addplot[black!30, thick, domain = 1:4]{ \mone*(x-\zero) };   
    \addplot[black!30, thick, domain = 4:5]{ -2.25+(x-5)*(x-5) };
    
     \pgfmathsetmacro{\zero}{2.5}

    \pgfmathsetmacro{\mone}{-2.93/\zero}
    \pgfmathsetmacro{\alfaone}{\mone / 2 *(1 - 2*\zero)}
    \pgfmathsetmacro{\betaone}{\mone / 2 }

    \addplot[black, thick, domain = 0:1]{ \alfaone + \betaone *x*x };
    \addplot[black, thick, domain = 1:4]{ \mone*(x-\zero) };   
    \addplot[black, thick, domain = 4:5]{ -1.76-.6+.6*(x-5)*(x-5) };

    \node[anchor = north west ] at (axis cs: 0,0) {$r_1 L$};
    \node[anchor = south east ] at (axis cs: 5,0) {$r_2 L$};
    \node[anchor = north east ] at (axis cs: 2.3,0) {$r_0 L$};
    
    \node[anchor = east ] at (axis cs: 0,2.85) {$\alpha$};
    \node[anchor = east ] at (axis cs: 0,1.425) {$\frac{\alpha}{2}$};

	\node[anchor = west ] at (axis cs: 5,-1.25) {$-\frac{d}{2}$};   
    \node[anchor = west ] at (axis cs: 5,-2.5) {$-d$};   
    
    \node[anchor = south west ] at (axis cs: 1.5,1.7) {$\nu_1 \alpha$};
    \node[anchor = north east ] at (axis cs: 3.5,-1.9) {$-\nu_2 d$};

\end{axis}
\end{tikzpicture}
\caption{Visual representation of the construction of $v$. In red, areas where $\mu_{1}^L=M_1$. In blue, areas where $\mu_{2}^L=M_2$. In gray, the bounds given by $\underline{\nu_L}$ and $\overline{\nu_L}$. In black, the solution $v$.}
\end{figure}
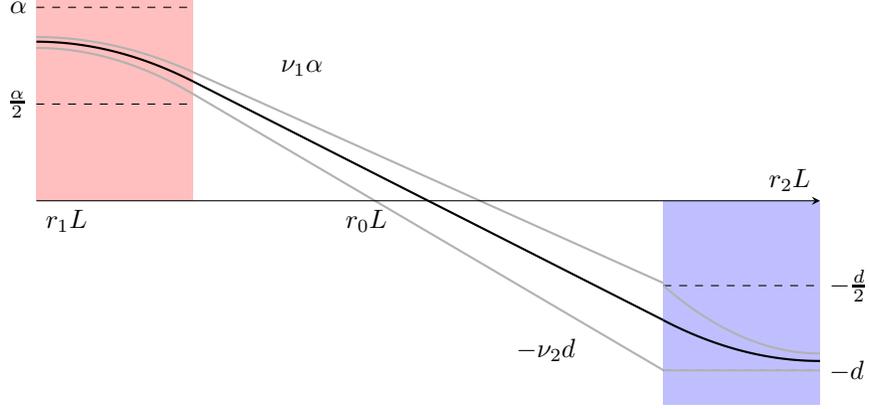

\begin{proof}
Let 
\[
\delta:\left(\nu,L\right)\mapsto -\Phi_1\left(\nu,L\right)r_0 L+\alpha\nu.
\]
The function $\nu\mapsto\delta\left(\nu,L\right)$ is, for all $L>0$, an 
increasing homeomorphism from $\left[\frac{1}{2},1\right)$ onto 
\[
\left[-\Phi_1\left(\frac{1}{2},L\right)r_0 L+\frac{\alpha}{2},\alpha\right).
\]

Since $L\mapsto-\Phi_1\left(\frac{1}{2},L\right)r_0 L$ is decreasing and goes to $-\infty$ as $L\to+\infty$, 
we can define the unique $L_0>0$ satisfying 
\[
-\Phi_1\left(\frac{1}{2},L_0\right)r_0 L_0 +\frac{\alpha}{2}=-d.
\]
Then for all $L>L_0$, we can define the unique $\underline{\nu_L}\in\left(\frac{1}{2},1\right)$ and the unique 
$\overline{\nu_L}\in\left(\underline{\nu_L},1\right)$ satisfying respectively
\[
\delta\left(\underline{\nu_L},L\right)=-d\textup{ and }\delta\left(\overline{\nu_L},L\right)=-\frac{d}{2}.
\]

Now let 
\[
\psi:\left(\nu,L\right)\mapsto\Phi_1\left(\nu,L\right)-\Phi_2\left(-\frac{\delta\left(\nu,L\right)}{d},L\right),
\]
well-defined in $\left(\underline{\nu_L},\overline{\nu_L}\right]$ for all $L>L_0$. For all $L>L_0$, 
$\nu\mapsto\psi\left(\nu,L\right)$ is a decreasing homeomorphism satisfying 
\[
\lim_{\nu\to\underline{\nu_L}} \psi\left(\nu,L\right) = \frac{\alpha\underline{\nu_L}+d}{r_0 L}>0,
\]
\[
\psi\left(\overline{\nu_L},L\right) = \frac{\alpha\overline{\nu_L}+\frac{d}{2}}{r_0 L}-\Phi_2\left(\frac{1}{2},L\right).
\]
Since $L\mapsto \psi\left(\overline{\nu_L},L\right)$ goes to $-\gamma_{d,M_2,\frac{1}{2}}<0$ 
as $L\to+\infty$, we can define $\underline{L}\geq L_0$ such that, for all $L>\underline{L}$, 
\[
\psi\left(\overline{\nu_{L}},L\right)<0
\]
and deduce that for all $L>\underline{L}$, there exists a unique $\nu_L\in\left(\underline{\nu_L},\overline{\nu_L}\right)$
satisfying $\psi\left(\nu_L,L\right)=0$, that is 
\[
\Phi_1\left(\nu_L,L\right)=\Phi_2\left(-\frac{\delta\left(\nu_L,L\right)}{d},L\right).
\]

Next, we fix $L>\underline{L}$ and define $w_1=w_{\alpha,M_1,\nu_L,r_1 L}$, 
$w_2=w_{d,M_2,-d^{-1}\delta\left(\nu_L,L\right),r_2 L}$ 
as well as the nonzero, sign-changing, $L$-periodic function $v$ by
\[
v_{|\left[0,L\right)}(x) = 
\begin{cases}
	w_1\left(x\right) & \textup{if }x\in\left[0,r_1 L\right) \\
    -\Phi_1\left(\nu_L,L\right)\left(x-r_1 L\right)+\nu_L\alpha & \textup{if }x\in\left[r_1 L,r_1 L+r_0 L\right) \\
    w_2\left(x-r_1 L-r_0 L-r_2 L\right) & \textup{if }x\in\left[r_1 L+r_0 L,r_1 L+r_0 L+2r_2 L\right) \\
    \Phi_1\left(\nu_L,L\right)\left(x-L+r_1 L\right)+\nu_L\alpha & \textup{if }x\in\left[r_1 L+r_0 L+2r_2 L,r_1 L+2r_0 L+2r_2 L\right) \\
    w_1\left(x-L\right) & \textup{if }x\in\left[r_1 L+2r_0 L+2r_1 L,L\right)
\end{cases}
\]

Since, by construction, $v$ is a $\mathscr{C}_{L\textup{-per}}^{1,1} \subset H_{L\textup{-per}}^{2}$ juxtaposition of piecewise solutions of (\ref{eq:normalized_elliptic_equation}), we readily deduce that it is a solution of (\ref{eq:normalized_elliptic_equation}).

Regarding the stability of the solution $v$, from \lemref{C1_solution_operator} we evince that the linearized elliptic operator at $v$, denoted $\mathscr{L} \in L \left(  H^{2}_{L\textup{-per}}, L^{2}_{L\textup{-per}}\right)$, is
\[
\mathscr{L}:\eta \mapsto \eta'' +\left[\mu_{1}\left(\alpha-2 v\right)\mathbf{1}_{v>0}+\mu_{2}\left(d+2v\right)\mathbf{1}_{v<0}\right]\eta. 
\]

First we verify the stability in the sense of (\ref{eq:parabolic_semilinear_equation}). Let $\lambda$ be the corresponding 
periodic principal eigenvalue and $\psi\in H_{L\textup{-per}}^{2}$ be the associated
unique periodic positive eigenfunction, normalized in $L^{2}\left(\left(0,L\right)\right)$. From the identity 
\[
	\int_0^L \left(-\mathscr{L}\psi - \lambda \psi \right) \psi = 0
\]
we deduce
\begin{align*}
	\int_{0}^{L}\left(\psi'\right)^{2} & =\int_{0}^{L}\left[\mu_{1}\left(\alpha- 2 v \right)\mathbf{1}_{v>0}+\mu_{2} \left(d + 2 v \right)\mathbf{1}_{v<0}\right]\psi^{2}+\lambda\\
	& =M_1\int_{\{\mu_1 > 0\}\cap\{v>0\}}\left(\alpha-2 v\right)\psi^{2}+M_2\int_{\{\mu_2 > 0\}\cap\{v<0\}}\left(d+2v\right)\psi^{2}+\lambda.
\end{align*}
Since by construction 
\[
v\geq\nu_L\alpha>\frac{\alpha}{2}\textup{ in }\{\mu_1 > 0\}\cap\{v>0\}
\]
and 
\[
v\leq-\left(-\frac{\delta\left(\nu_L,L\right)}{d}\right)d<-\frac{d}{2}\textup{ in }\{\mu_2 > 0\}\cap\{v<0\},
\]
we deduce
\[
\lambda>\int_{0}^{L}\left(\psi'\right)^{2}>0.
\]

Similarly, we verify the stability of $v$ in the sense of (\ref{eq:parabolic_quasilinear_equation}). The same computations as 
before lead us to the desired conclusion.

This conclude the proof of existence and stability of sign-changing solutions for piecewise-constant coefficients
\end{proof}

\begin{rem*}
Going carefully through the proof, using $\overline{\nu_L}<1$ and assuming that $\underline{L}$ is minimal, 
we obtain the estimate $\underline{L}<L^\star$, where $L^\star>0$ is the unique solution of
\[
\Phi_2\left(\frac{1}{2},L^\star\right)L^\star=\frac{1}{r_0}\max\left(\alpha+\frac{d}{2},\frac{\alpha}{2}+d\right).
\]
Hence estimating $\underline{L}$ is only a matter of estimating $L\mapsto\Phi_2\left(\frac{1}{2},L\right)$.
Unfortunately, being unable to find any satisfying estimation of $\Phi_2$, we do not pursue further. 
\end{rem*}

\subsubsection{With regular coefficients}
The function $v$ constructed in \propref{Existence_for_the_scalar_eq_in_L_infty} is linear around $v = 0$. Thus there exists an open neighborhood $O \subset H_{L\textup{-per}}^{2}$ satisfying the assumptions of \lemref{C1_solution_operator}. 

\begin{prop}
\label{prop:In_a_L_infty_neighborhood} Under the assumptions of \propref{Existence_for_the_scalar_eq_in_L_infty}, for any $L > \underline{L}$ there exists an open neighborhood $U \subset \left(L_{L\textup{-per}}^{\infty}\right)^{2}$ of $(\mu_1, \mu_2)$ such that
for all $\left(\rho_{1},\rho_{2}\right)\in U$, (\ref{eq:normalized_elliptic_equation}) with $\left(\rho_{1},\rho_{2}\right)$
admits a sign-changing, $L$-periodic, weak solution. The solution is also linearly stable in the sense of 
(\ref{eq:parabolic_semilinear_equation}) and (\ref{eq:parabolic_quasilinear_equation}). 
\end{prop}

\begin{proof}
Let $L > \underline{L}$ and let $(\mu_1, \mu_2, v) \in \left(L_{L\textup{-per}}^{\infty}\right)^{2} \times  H_{L\textup{-per}}^{2}$ be the solution constructed in \propref{Existence_for_the_scalar_eq_in_L_infty}. 

The prerequisites of the implicit function theorem are readily satisfied for the functional $\mathscr{F}$ at $(\mu_1, \mu_2, v)$. 
In particular, since the solution $v$ is linearly stable  in the sense of (\ref{eq:parabolic_semilinear_equation}), the functional 
$\frac{\partial\mathscr{F}}{\partial z}\left[\mu_{1},\mu_{2},v\right]$ is invertible in the following sense: for all 
$f\in  L_{L\textup{-per}}^{2}$, there exists a unique weak solution $z_{f}\in H_{L\textup{-per}}^{2}$ of
\[
	\frac{\partial\mathscr{F}}{\partial v}\left[\mu_{1},\mu_{2},z\right]\left(z_{f}\right) = f.
\]
This follows by standard regularity results.

By virtue of the implicit function theorem, there exists an open neighborhood $U \subset \left(L_{L\textup{-per}}^{\infty}\right)^{2}$ of $\left(\mu_{1},\mu_{2}\right)$, an open neighborhood $V \subset O \subset H_{L\textup{-per}}^{2}$ of $v$ and a $\mathscr{C}^{1}$ diffeomorphism $\Psi:U\to V$ such that, for all $\left(\rho_{1},\rho_{2}\right)\in U$,
\[
\mathscr{F}\left[\rho_{1},\rho_{2},\Psi\left[\rho_{1},\rho_{2}\right]\right]=0.
\]
Finally, since the map $\Psi$ is $\mathscr{C}^{1}$, we find that the linear stability of the solution is preserved in a open neighborhood of $(\mu_1, \mu_2)$.
\end{proof}

\section{The strongly competitive competition\textendash diffusion system}

In the previous section we have considered the equation
\begin{equation}\label{eq:sign_chang}
    -z'' = \frac{\mu_1}{\alpha} \left(\alpha - z\right)z^+ - \frac{\mu_2}{d^2} \left(d + z \right) z^-.
\end{equation}
For this equation and particular choices of $\mu_1$ and $\mu_2$, we have constructed a sign-changing 
solution $v \in \mathscr{C}^{1,1}_{L\textup{-per}}$ for periods $L$ greater than a threshold $\underline{L}$. 
We have also shown that this solution is linearly stable in the sense of (\ref{eq:parabolic_semilinear_equation}) and 
(\ref{eq:parabolic_quasilinear_equation}).

In this section, we aim at using this result to prove the existence of linearly stable solutions of (\ref{eq:elliptic_system}). 
Specifically, fixing $L>\underline{L}$ and a positive $L$-periodic smooth 
function $\omega$, our aim is to prove that for any $k> 0$ large enough there exists a positive and stable solution of (\ref{eq:elliptic_system}) $(u_{1,k},u_{2,k}) \in \mathscr{C}^{1,1}_{L\textup{-per}}$ such that
\[
  (u_{1,k},u_{2,k}) \to \left(\frac{v^+}{\alpha}, \frac{v^-}{d}\right) \qquad \textup{as } k\to+\infty
\]
in $H^1_{L\textup{-per}}$ and $\mathscr{C}^{0,\gamma}_{L\textup{-per}}$ for $\gamma \in (0,\frac{1}{2})$. 

We will show the result in a series of steps: first, we give some \textit{a priori} estimates of the solution of a more general class of systems. Then, by means of topological arguments, we deduce from these estimates the existence of solutions. Finally we establish the uniqueness and the linear stability of the solutions.

\subsection{\textit{A priori} estimate}

We start by showing \textit{a priori} estimates for the solutions of a family of systems that contains 
(\ref{eq:elliptic_system}) as a special case. We are here interested in the $L$-periodic positive solutions of
\begin{equation}\label{eq:homotopy_a_priori}
\begin{cases}
    -u_1'' = t \mu_1( 1 - u_1) u_1 + (1-t) \frac{\mu_1}{\alpha^2} \left(\alpha - (\alpha u_{1} - d u_2)^+\right)(\alpha u_{1} - d u_2)^+ - k \omega u_1 u_2 \\
    -d u_2'' = t \mu_2( 1 - u_2) u_2 + (1-t) \frac{\mu_2}{d^2} \left(d - (\alpha u_{1} - d u_2)^-\right)(\alpha u_{1} - d u_2)^- - \alpha k \omega u_1 u_2
  \end{cases}
\end{equation}
where $k > 0$ and $t \in [0,1]$. Observe that if we take $t = 1$, then (\ref{eq:homotopy_a_priori}) reduces to the original system (\ref{eq:elliptic_system}).

\begin{lem}\label{lem:a_priori_infty}
Let $\eta > 0$. There exists a constant $C > 0$ such that for any $t \in [0,1]$ and $k \geq 1$, if $(u_1, u_2)$ is a nonnegative nonzero solution of (\ref{eq:homotopy_a_priori}) and
\[
	\|(\alpha u_{1} - d u_2) - v\|_{L^\infty} \leq \eta
\]
then
\[
	0 < u_1 < C, \quad 0 < u_2 < C \quad \textup{and} \quad \|(u_1,u_2)\|_{\textup{Lip}} \leq C.
\]
\end{lem}
\begin{proof}
We start by showing that nonnegative nonzero solutions are necessarily strictly positive. Indeed, assuming that $u_2 \geq 0$, we have that $0$ is a solution of the equation in $u_1$, since in this case $(\alpha 0 - d u_2)^+ = 0$. We thus conclude by the comparison principle that $u_1 > 0$. 

In order to show the upper uniform bound, we first observe that by assumption
\[
	du_2 \geq \alpha u_1 + v - \eta
\]
and that, moreover, there exists a constant $C' > 0$ such that
\[
	t \mu_1( 1 - u_1) u_1 + (1-t) \frac{\mu_1}{\alpha^2} \left(\alpha - (\alpha u_{1} - d u_2)^+\right)(\alpha u_{1} - d u_2)^+ \leq C'.
\]
As a result, any $u_1$ positive solution of (\ref{eq:homotopy_a_priori}) satisfies the differential inequality
\[
	-u_1'' \leq C' - \frac{k}{d} \omega u_1 (\alpha u_1 + v - \eta).
\]
It follows that any maximum $M>0$ of $u_1$ satisfies 
\[
    \frac{k}{d} \omega M (\alpha M + v - \eta) \leq C',
\]
whence $u_1$ is bounded by some constant $C > 0$. We can conclude similarly for the component $u_2$.

To prove the uniform Lipschitz estimate, we integrate the equation in $u_1$ on the interval $[0,L]$. Exploiting the $L$-periodicity of $u_1$, we find
\begin{equation}\label{eq:bound_compet}
	k \int_0^L \omega u_1 u_2 = \int_0^L t \mu_1( 1 - u_1) u_1 + \int_0^L (1-t) \frac{\mu_1}{\alpha^2} \left(\alpha - (\alpha u_{1} - d u_2)^+\right)(\alpha u_{1} - d u_2)^+.
\end{equation}
Once again, the right hand side is bounded by $C' L$ for any $t \in [0,1]$ and $k \geq 1$. Since $u_1$ is periodic and smooth ($\mathscr{C}^{1,1}$)  for $k$ bounded, there exists $x_0 \in [0,L]$ such that $u_1'(x_0)= 0$. Integrating the equation in $u_1$ on the interval $[x_0, x]$ We find that
\[
    u_1'(x) = - \int_{x_0}^{x} \left[ t\mu_1( 1 - u_1) u_1 + (1-t) \frac{\mu_1}{\alpha^2} \left(\alpha - (\alpha u_{1} - d u_2)^+\right)(\alpha u_{1} - d u_2)^+ \right] + k \int_{x_0}^{x} \omega u_1 u_2
\]
which yields, together with (\ref{eq:bound_compet}), the estimate for any $x \in [0,L]$
\[
	|u_1'(x)| \leq 2 \int_{0}^{L} t \mu_1| 1 - u_1| u_1 + 2 \int_{0}^{L} (1-t) \frac{\mu_1}{\alpha^2} \left|\alpha - (\alpha u_{1} - d u_2)^+\right|(\alpha u_{1} - d u_2)^+.
\]
We conclude that the component $u_1$ is bounded in the Lipschitz norm uniformly in $t \in [0,1]$ and $k \geq 1$. We can proceed in a similar way for the component $u_2$. 
\end{proof}

\begin{lem}\label{lem:a_priori}
Let $\eta > 0$ be sufficiently small. For any $\eps > 0$ there exists $\bar k \geq 1$ such that any nonnegative solution $(u_1, u_2)$ of (\ref{eq:homotopy_a_priori}) with $k \geq \bar k$ such that
\[
	\|(\alpha u_{1} - d u_2) - v\|_{L^\infty} \leq \eta
\]
satisfies
\[
  \left\| (u_{1}, u_{2}) - \left(\frac{v^+}{\alpha}, \frac{v^-}{d}\right)\right\|_{H^1_{L\textup{-per}} \cap \mathscr{C}^{0,\frac12}}  + \left\| \alpha u_{1} - d u_{2} - v \right\|_{\mathscr{C}^{1,\frac12}} \leq \eps.
\]
\end{lem}
\begin{proof}
By the uniform Lipschitz estimate of \lemref{a_priori_infty} and the Ascoli--Arzela theorem, we find that the set of solutions in the statement is compact in the $\mathscr{C}^{0,\gamma}$ topology for any $\gamma \in [0,1)$ and limit points are Lipschitz continuous. Let $(\bar u_1, \bar u_2) \in \textup{Lip}_{L\textup{-per}}$ be the limit of a converging sequence of solutions $((u_{1,k}, u_{2,k}))_k$ as $k \to +\infty$. Integrating the equation in $u_{1,k}$ over $[0,L]$ and taking the limit $k \to +\infty$ (see also the identity in (\ref{eq:bound_compet})), we find that $\bar u_1 \bar u_2 = 0$ must be satisfied. In particular, it follows that
\begin{equation}\label{eq:limit_segr}
	(\alpha \bar u_1 - d \bar u_2)^+ = \alpha \bar u_1 \quad \textup{and} \quad (\alpha \bar u_1 - d \bar u_2)^- = d \bar u_2.
\end{equation}
Moreover, since the function $v$ changes sign in $[0,L]$, by taking $\eta > 0$ small enough, we find that $\bar u_1$ and $\bar u_2$ cannot be identically zero. Testing the equation in $u_{1,k}$ by $u_{1,k}$ itself, we find
\begin{multline*}
    \int_0^L (u_{1,k}')^2 + k \omega u_{1,k}^2 u_{2,k}  =  \int_0^L t \mu_1( 1 - u_{1,k}) u_{1,k}^2 \\
    + \int_0^L (1-t) \frac{\mu_1}{\alpha^2} \left(\alpha - (\alpha u_{1,k} - d u_{2,k})^+\right)(\alpha u_{1,k} - d u_{2,k})^+ u_{1,k}
\end{multline*}
from which we obtain that the sequence $(u_{1,k})_k$ is bounded in $H^1_{L\textup{-per}}$. By the compact embedding of $H^1_{L\textup{-per}}$ in $L^2_{L\textup{-per}}$ we also obtain that $(u_{1,k})_k$ converges to $\bar u_1$ weakly in $H^1_{L\textup{-per}}$. Testing now the equation by $u_{1,k}- \bar u_1$ and using (\ref{eq:bound_compet}) to bound the coupling term as in the proof of \lemref{a_priori_infty}, we obtain
\[
    \begin{split}
	\int_0^L \left[\left(u_{1,k}-\bar u_1\right)'\right]^2  \leq &- \int_0^L \bar u_1' \left(u_{1,k}-\bar u_1\right)' + 2\sup_{[0,L]}\left|u_{1,k}-\bar u_1\right| \left[ \int_{0}^{L} t \mu_1| 1 - u_{1,k}| u_{1,k} \right. \\ 
	&\left. + \int_{0}^{L} (1-t) \frac{\mu_1}{\alpha^2} \left|\alpha - (\alpha u_{1,k} - d u_{2,k})^+\right|(\alpha u_{1,k} - d u_{2,k})^+ \right].
    \end{split}
\]
As a result, the sequence $(u_{1,k})_k$ converges to $\bar u_1$ also strongly in $H^1_{L\textup{-per}}$. Similar conclusions hold for the sequence $(u_{2,k})_k$.

We now consider the equation verified by $\alpha u_{1,k} - d u_{2,k}$. We find
\begin{multline*}
    -(\alpha u_{1,k} - d u_{2,k})'' = \alpha t \mu_1( 1 - u_{1,k}) u_{1,k} - t \mu_2( 1 - u_{2,k}) u_{2,k} \\
    + (1-t) \frac{\mu_1}{\alpha} \left(\alpha - (\alpha u_{1,k} - d u_{2,k})^+\right)(\alpha u_{1,k} - d u_{2,k})^+ \\
    - (1-t) \frac{\mu_2}{d^2} \left(d - (\alpha u_{1,k} - d u_{2,k})^-\right)(\alpha u_{1,k} - d u_{2,k})^-.
\end{multline*}
Passing to the limit in the equation and exploiting (\ref{eq:limit_segr}), we obtain
\begin{multline*}
    -(\alpha \bar u_1 - d \bar u_2)'' = \frac{\mu_1}{\alpha} \left(\alpha - (\alpha \bar u_{1} - d \bar u_{2})^+\right)(\alpha \bar u_{1} - d \bar u_{2})^+ \\
    - \frac{\mu_2}{d^2} \left(d - (\alpha \bar u_{1} - d \bar u_{2})^-\right)(\alpha \bar u_{1} - d \bar u_{2})^-
\end{multline*}
That is, the function $\alpha \bar u_1 - d \bar u_2$ is a solution of (\ref{eq:sign_chang}) that is $\eta$-close, in $L^\infty$ topology, to the solution $v$. Since $v$ is an isolated solution of the equation, by taking $\eta$ sufficiently small we find that necessarily $\alpha \bar u_1 - d \bar u_2 = v$. As this is true for any sequence of converging solutions $((u_{1,k},u_{2,k}))_k$, we find the sought conclusion.
\end{proof}

An interesting consequence of the previous result is that the solutions of (\ref{eq:homotopy_a_priori}), when $\eta$ is small and $k$ is large , are close to the segregated state $\left(\frac{v^+}{\alpha}, \frac{v^-}{d}\right)$, independently of the value of $t \in [0,1]$. More precisely, we have the following corollary.

\begin{cor}\label{cor:closedness}
There exists $\eta_1 > 0$ such that for any $\eps > 0$, $t \in [0,1]$ and $k \geq \bar k(\eps) > 0$, if $(u_1,u_2) \in \mathscr{C}^{1,1}_{L\textup{-per}}$ is a solution of (\ref{eq:homotopy_a_priori}) such that
\[
	\|(\alpha u_{1} - d u_2) - v\|_{L^\infty} < \eta_1
\]
then
\[
	\left\|(u_1,u_2) - \left(\frac{v^+}{\alpha}, \frac{v^-}{d}\right) \right\|_{H^1_{L-\textup{per}} \cap \mathscr{C}^{0,\frac12}} \leq \eps.
\]
\end{cor}

\subsection{Existence of solutions}

We now show the existence of solution of (\ref{eq:elliptic_system}) when $k$ is large. We will prove this result in two steps, first proving the existence of solutions of the auxiliary problem when $t=0$, and then, making use of a homotopy argument, we will transfer this result to the original problem. Our argument is inspired by the method proposed in \cite{Dancer_Du_1994} to prove the existence of solutions of a related problem.

\begin{lem}\label{lem:decoupled}
There exists $\eta_2>0$ such that, for any $k > 0$, there exists a unique positive solution $(u_1,u_2) \in \mathscr{C}^{1,1}_{L\text{-per}}$ of 
\begin{equation}\label{eq:homotopy_at_0}
\begin{cases}
    -u_1'' = \frac{\mu_1}{\alpha^2} \left(\alpha - (\alpha u_{1} - d u_2)^+\right)(\alpha u_{1} - d u_2)^+ - k \omega u_1 u_2 \\
    -d u_2'' = \frac{\mu_2}{d^2} \left(d - (\alpha u_{1} - d u_2)^-\right)(\alpha u_{1} - d u_2)^- - \alpha k \omega u_1 u_2
  \end{cases}
\end{equation}
satisfying  
\[
\|\alpha u_1 - d u_2 - v\|<\eta_2. 
\]

This solution is $L$-periodic and linearly stable.
\end{lem}

\begin{proof}
First, we claim that there exists $\eta_2>0$ so small that solutions satisfying the preceding assumptions verify in fact the identity $\alpha u_1 - d u_2 = v$. Indeed, combining the two equations in (\ref{eq:homotopy_at_0}) we find that $\alpha u_1 - d u_2$ is a solution of (\ref{eq:sign_chang}) that is also close to $v$ in the $L^\infty$ topology. Since $v$ is a stable, whence isolated, solution of (\ref{eq:sign_chang}) , necessarily $\alpha u_1 -d u_2 = v$.

We proceed by showing that there exists a unique pair $(u_1,u_2)$ in the class of all $(u_1,u_2)$ satisfying $\alpha u_1 -d u_2 = v$. We notice that in the set of all $(u_1,u_2) \in \mathscr{C}^{1,1}_{L\text{-per}}$ satisfying  $\alpha u_1 - d u_2 = v$, the two equations
of (\ref{eq:homotopy_at_0}) are equivalent. Indeed, assuming $\alpha u_1 - d u_2 = v$, 
\begin{align*}
    \alpha\left(u_1''+\frac{\mu_1}{\alpha^2} \left(\alpha - v^+\right)v^+  - k \omega u_1 u_2\right) & = v''+d u_2''+\frac{\mu_1}{\alpha} \left(\alpha - v\right)v^+ - \alpha k \omega u_1 u_2 \\
    & = \frac{\mu_2}{d^2} \left(d + v \right) v^- +d u_2'' -\alpha k \omega u_1 u_2 \\
    & = d u_2''+\frac{\mu_2}{d^2} \left(d - v^- \right) v^- -\alpha k \omega u_1 u_2
\end{align*}

Therefore it suffices to prove the existence, uniqueness and linear stability of 
$u \in \mathscr{C}^{1,1}_{L\text{-per}}$ such that 
\begin{equation}\label{eq:homotopy_at_0_uncoupled}
    -u'' = \frac{\mu_1}{\alpha^2} \left(\alpha - v^+\right)v^+ +\frac{k \omega}{d} u(v-\alpha u).
\end{equation}

Notice as a preliminary that, up to the forcing
term $\frac{\mu_1}{\alpha^2} \left(\alpha - v^+\right)v^+ \geq 0$, this equation falls in the general
theory of periodic KPP reaction--diffusion equations developed by Berestycki, Hamel and Roques in 
\cite{Berestycki_Ham_1}. 

On one hand, $\frac{v^+}{\alpha}$ is a nonnegative nonzero sub-solution for
(\ref{eq:homotopy_at_0_uncoupled}). On the other hand, any sufficiently large constant is a super-solution.
The existence of a bounded positive solution $u$ satisfying $\alpha u>v^+$ follows. The uniqueness 
is easily established thanks to a classical comparison argument
relying upon the logistic form of $u\mapsto\frac{k \omega}{d} u(v-\alpha u)$ 
(we refer for instance to Berestycki--Hamel--Roques \cite[Theorem 2.4]{Berestycki_Ham_1}; regarding
uniqueness, the forcing term $\frac{\mu_1}{\alpha^2} \left(\alpha - v^+\right)v^+ $ does not play any role).
The periodicity then follows directly from the uniqueness. Finally, by definition, the 
solution $u$ is linearly stable if 
\[
    \lambda_{1,L-\textup{per}}\left(-\frac{\textup{d}^2}{\textup{d}x^2}-\frac{k \omega}{d}\left(v-2\alpha u\right)\right)>0.
\]
It is well-known that the preceding inequality is satisfied if $v-2\alpha u < 0$, which is true indeed since
$\frac{v}{2} \leq v^+ < \alpha u$.
\end{proof}

We now pass to the second step of the construction. 
For notation convenience, let 
$X = \mathscr{C}^{0,1/2}_{L\text{-per}}$ 
(any H\"older exponent $\gamma \in \left(0,1\right)$ would do) and let 
$L \in \mathcal{K}(X;X)$ be the linear 
compact operator such that, for all $z, f \in X$, $z=Lf$ if and only if $-z'' + z = f$. 

We consider the homotopy $H : X^2 \times [0,1] \to X^2$ defined by 
\[
    H(u;t)=u-L\left(u+f(u;t)\right),
\]
where 
\[
  f\left(u; t\right) = \begin{pmatrix}
    t \mu_1( 1 - u_1) u_1 + (1-t) \frac{\mu_1}{\alpha^2} \left(\alpha - (\alpha u_{1} - d u_2)^+\right)(\alpha u_{1} - d u_2)^+ - k \omega u_1 u_2 \\
    \frac{1}{d}\left(t \mu_2( 1 - u_2) u_2 + (1-t) \frac{\mu_2}{d^2} \left(d - (\alpha u_{1} - d u_2)^-\right)(\alpha u_{1} - d u_2)^- - \alpha k \omega u_1 u_2\right)
  \end{pmatrix}.
\]
Observe that the homotopy $H$ is of the form $\id-K_t$ where $\id : X^2 \to X^2$ is the identity operator, and $K_t \in \mathcal{K}(X^2\times [0,1];X^2)$ is a compact operator for any $t \in [0,1]$ and is continuous in $t$, by standard elliptic estimates. In this regard, we observe that $k$ is fixed.

We have that $H(u_1,u_2; 0) = 0$ if and only if $(u_1,u_2)$ is a solution of (\ref{eq:homotopy_at_0}), while $H(u_1,u_2; 1) = 0$ if and only if $(u_1,u_2)$ is a solution of (\ref{eq:elliptic_system}). Our goal is to apply the theory of the Leray--Schauder degree in order to evince the existence of solutions of (\ref{eq:elliptic_system}) from the existence of solutions of (\ref{eq:homotopy_at_0}), \lemref{decoupled}.

Now, we fix $\eta=\min\left(\eta_1,\eta_2\right)$ (see \corref{closedness} and \lemref{decoupled}) and define, for any $\eps>0$, the set 
\[
  	{O}_{\eps}= \left\{u\in X^2\ |\ u_1>0,\ u_2>0,\ \left\|\alpha u_{1} - d u_2 - v\right\|_{L^\infty} < \eta,\ \left\|u - \left(\frac{v^+}{\alpha}, \frac{v^-}{d}\right) \right\|_{X^2} < 2\eps\right\}.
\]
It is a connected open subset of $X^2$. Moreover, it should be noticed that provided $\varepsilon$ 
is small enough, then ${O}_\varepsilon$ does not depend on $\eta$ and reduces to 
\[
  	{O}_{\eps}= \left\{u\in X^2\ |\ u_1>0,\ u_2>0,\ \left\|u - \left(\frac{v^+}{\alpha}, \frac{v^-}{d}\right) \right\|_{X^2} < 2\eps\right\}.
\]

\begin{lem}
For any $\eps > 0$ there exists $\bar k > 0$ such that the equation
\[
  H(u_1,u_2;t) = 0
\]  
has no solutions for any $t\in [0,1]$ and $k \geq \bar k$ on $\partial {O}_\eps$.
\end{lem}
This result follows directly from \corref{closedness}.

\begin{lem} For any $\eps>0$, the equation
\[
  H(u_1,u_2; 0) = 0
\]
has a unique solution in ${O}_\eps$. Moreover there exists $\bar k > 0$ such that if $k \geq \bar k$, then this solution has fixed point index 1, that is
\[
  \ind_{X^2}({O}_\eps; (u_1,u_2)) = 1.
\]
\end{lem}
This result follows from \lemref{decoupled}. We also recall that the fixed point index of an isolated solution can be computed by linearization if the equation involves $\mathscr{C}^1$ operators, \cite[Theorem 4.2.11]{Ambrosetti_Arc}.

We can thus conclude by virtue of the Leray--Schauder theorem (see \cite{Leray_Schauder} and \cite[Theorem 4.3.4]{Ambrosetti_Arc}).

\begin{lem}\label{lem:u1_u2_finally}
For any $\eps > 0$, there exists $\bar k > 0$ such that, for all $k > \bar k$, (\ref{eq:elliptic_system}) has a solution $(u_{1,k},u_{2,k})$ in $\mathcal{O}_\eps$. Moreover,
\[
  	\lim_{k \to +\infty}  \left\| (u_{1,k}, u_{2,k}) - \left(\frac{v^+}{\alpha}, \frac{v^-}{d}\right)\right\|_{H^1_{L\textup{-per}} \cap \mathscr{C}^{0,\frac12}}  + \left\| \alpha u_{1,k} - d u_{2,k} - v \right\|_{\mathscr{C}^{1,\frac12}} = 0.
\]
\end{lem}

If needed, one can improve the convergence result, by stating that the solutions are uniformly bounded in the Lipschitz norm and converge in the $\mathscr{C}^{0,\gamma}$ norm for any $\gamma \in (0,1)$. See, on this subject, the results in \cite{Conti_Terracin}. 

\subsection{Linear stability for $k$ large}

We now investigate the linear stability of the solutions obtained in \lemref{u1_u2_finally}. 
To this end, we consider the linearized system (\ref{eq:elliptic_system}) at the solution $(u_1, u_2)$ and introduce its periodic 
principal eigenvalue. 

For all $k>\bar k$, let
\[
\lambda_{1,k} =\lambda_{1,\textup{$L$-per}}\left(-\left(\begin{matrix}\frac{\mathrm{d}^{2}}{\mathrm{d}x^{2}}+\mu_{1}\left(1-2u_{1,k}\right)-k\omega u_{2,k} & k\omega u_{1,k}\\
\alpha k\omega u_{2,k} & d\frac{\mathrm{d}^{2}}{\mathrm{d}x^{2}}+\mu_{2}\left(1-2u_{2,k}\right) -\alpha k\omega u_{1,k}
\end{matrix}\right)\right)
\]
and assume that the associated periodic principal eigenfunction $(\varphi_k, \psi_k)$ is normalized in such a way that
\[
  \max_{x \in \left[0,L\right]} \left(\alpha\varphi_k + d\psi_k \right)(x) = 1.
\]
Observe that since both $\varphi_k$ and $\psi_k$ are positive, this automatically implies that the two functions are globally bounded.

We start by showing a priori estimates on the principal eigenvalue and the principal eigenfunctions.
\begin{lem}\label{lem:bounded_eigen}
The principal eigenvalues are uniformly bounded from below. There exists $C\in \R$ such that
\[
  \lambda_{1,k} > -C \qquad \textup{for all }k>\bar k.
\]
\end{lem}
\begin{proof}
It suffices to take
\[
  C = \sup_{k>\bar k, x \in \R} \left(\left|\mu_1 (1 - 2u_{1,k})\right| + \left|\mu_2(1-2u_{2,k}) \right|\right).
\]
Indeed, the solution $(u_{1,k},u_{2,k}) \in {O}_\eps$ are uniformly bounded. Thus $C$ is finite. We then consider the sum of the equation in $\alpha \varphi_k$ and in $\psi_k$. The conclusion follows from the fact that the equation
\[
-\left(\alpha\varphi_k + d\psi_k \right)'' = \mu_1 (1 - 2u_{1,k}) \alpha \varphi_k + \mu_2( 1 - 2 u_{2,k})\psi_k + \lambda_{1,k} \left(\alpha\varphi_k + \psi_k\right),
\]
where the right-hand side is smaller than or equal to $(C+\lambda_{1,k}) \left(\alpha\varphi_k + \psi_k\right)$,
has no positive $L$-periodic solution if $\lambda_{1,k} < -C$.
\end{proof}

\begin{lem}\label{lem:decay_eigen}
For any $\eps > 0$ and $\delta > 0$, there exists $\bar k > 0$ such that
\[
  \sup_{\{v^- > \eps\}} \varphi_k +  \sup_{\{v^+ > \eps\}} \psi_k \leq \delta
\]
for any $k \geq \bar k$.
\end{lem}
\begin{proof}
We prove only the estimate in $\psi_k$, since the estimate in $\varphi_k$ follows the same reasoning. From now on, $\eps > 0$ and $\delta$ are fixed and we wish to show that 
\[
\sup_{\{v^+ > \eps\}} \psi_k \leq \delta.
\]

First, we observe that, since $v \in \mathscr{C}^{1,1}$, the constant 
\[
\ell=\frac{1}{4\|{v^+}'\|_{L^\infty}}> 0
\]
satisfies
\[
  \begin{split}
    \{v^+ > \eps\} + (-\ell \eps, \ell \eps) &\subset  \{v^+ > \eps / 2\} \textup{ and}\\
    \{v^+ > \eps / 2\} + (-\ell \eps, \ell \eps) &\subset \{v^+ > \eps / 4\}.
  \end{split}
\]  By uniform convergence of the sequence $(u_{1,k})_k$ to $v^+ / \alpha$, we have that, for $k$ large enough, $u_{1,k} > \frac{\eps}{8\alpha}$ on $\{v^+ > \eps / 4\}$. We now consider the equation  satisfied by $u_{2,k}$. We find that
\[
  \begin{split}
	  -d u_{2,k}''&=\mu_{2}\left(1-u_{2,k}\right)u_{2,k}-\alpha k\omega u_{1,k}u_{2,k} \\
	  &\leq \left[ \frac12 \|\mu_2\|_{L^\infty} - \inf_{x\in[0,L]} \omega(x)  k \frac{\eps}{8} \right] u_{2,k} \leq - A k \eps u_{2,k}  \qquad \textup{on $\{v^+ > \eps / 4\}$}
  \end{split}
\]
with a positive constant $A$ that can be chosen independently of $k$ and $\eps$ whenever $k$ is sufficiently large. 

Observe that the function $S:x\mapsto \beta \cosh(\sqrt{A k \eps/ d }x) $, $\beta > 0$, is a super-solution of the previous differential inequality and that $u_{2,k} \leq 1$ in $[0,L]$. Thus, choosing $\beta$ in such a way that $S(x) \leq 1$ for $x \in (-\ell \eps, \ell \eps)$, through a simple covering argument, the comparison principle yields
\[
  u_{2,k}(x) \leq 2 e^{-\sqrt{A k \eps^3 / d} \ell } \qquad  \textup{for all $x \in \{v^+ > \eps / 2\}$}.
\]
Finally, by the previous estimates, we deduce
\[
  \begin{split}
    -d \psi_k'' &= \alpha k \omega u_{2,k} \varphi_k + \left[\mu_2(1-2u_{2,k}) + \lambda - \alpha k \omega u_{1,k}\right] \psi_k \\
    &\leq B k e^{-\sqrt{A k \eps^3 / d} \ell } - C k \eps \psi_k   \qquad \textup{on $\{v^+ > \eps / 2\}$}
  \end{split}
\]
where, as before, the constants $B$ and $C$ can be chosen  independently of $k$ and $\eps$ whenever $k$ is sufficiently large. We can make use again a comparison with a super-solution, see \cite[Lemma 2.2]{Soave_Zilio_2015}, and conclude that
\[
  C k \eps \psi_{k}(x) \leq \frac{D}{d\ell^2} + B  k e^{-\sqrt{A k \eps^3 / d} \ell } \qquad  \textup{for all $x \in \{v^+ > \eps \}$}
\]
for $D$ universal positive constant. The result follows by taking $k$ large enough.
\end{proof}

With the uniform estimates of \lemref{bounded_eigen} and \lemref{decay_eigen} we are now in position to show that the 
solution $(u_1,u_2)$ constructed in the previous section is indeed linearly stable if $k$ is sufficiently large.

Of course, if $\liminf\limits_{k\to+\infty} \lambda_{1,k}=+\infty$, then the proof is done. Hence we assume from now on
that $\liminf\limits_{k\to+\infty} \lambda_{1,k}<+\infty$. Up to extraction of a subsequence, we also assume that 
$\lambda_{1,k}\to\liminf\limits_{k\to+\infty} \lambda_{1,k}$ as $k\to+\infty$. In particular, $\left(\lambda_{1,k}\right)_k$ is 
bounded. 

\begin{lem}\label{lem:conv_eign}For all $k>\bar k$, we define $Z_k \in \mathscr{C}^{1,1}_{L\textup{-per}}$ as
\[
  Z_k = \alpha\varphi_k + d\psi_k.
\]
Then the sequence of positive functions $\left(Z_k\right)_k$ is uniformly bounded in $W^{2,p}_{L\textup{-per}}$ and $\mathscr{C}^{1,\gamma}_{L\textup{-per}}$ for any $p < \infty$ and $\gamma < 1$. Each $Z_k$ solves
\[
    - Z_k'' = \left[\mu_1\left(1 - 2\frac{v^+}{\alpha}\right) + \frac{1}{d}\mu_2\left(1  + 2\frac{v^-}{d}\right)\right] Z_k + \lambda_{1,k} \sigma(v) Z_k + o_k(1)\\
\]
where $o_k(1)$ is a sequence of functions, bounded uniformly in $L^\infty$ and such that $o_k(1) \to 0$ in $L^p_{L\textup{-per}}$ for any $p < \infty$. 
\end{lem}
\begin{proof}
Once again, we take the sum of the equation in $\alpha \varphi_k$ and the equation in $\psi_k$. We thus find
\begin{equation}\label{eq:zn}
  -\left(\alpha\varphi_k + d \psi_k\right)'' = \mu_1 \left(1 - 2 u_{1,k}\right) \alpha\varphi_k + \mu_2\left(1 - 2 u_{2,k}\right) \psi_k + \lambda_{1,k} \left(\alpha\varphi_k + \psi_k\right).
\end{equation}
We observe that the terms in the right hand side of (\ref{eq:zn}) are uniformly bounded. Thus 
there exists $Z \in (H^{2} \cap \mathscr{C}^{1,\gamma})_{L\textup{-per}}$ such that, up to subsequence, $Z_k \to Z \geq 0$. By uniform convergence we have $\max Z = 1$. As a consequence of \lemref{decay_eigen}, we also have that
\[
	\left(\alpha\varphi_k + \psi_k\right) \to \left( \mathbf{1}_{v>0} + \frac{1}{d} \mathbf{1}_{v< 0} \right) Z = \sigma(v) Z
\]
in $L^p$ for any $p < \infty$. 

We now rearrange the terms of (\ref{eq:zn}) as follows:
\begin{multline*}
    -Z_k'' = \left[\mu_1\left(1 - 2 \frac{v^+}{\alpha}\right) + \frac{1}{d}\mu_2\left(1 + 2 \frac{v^-}{d}\right)\right] Z_k + \lambda_{1,k} \sigma(v) Z_k \\
    + \lambda_{1,k} \left[ \left(\alpha\varphi_k + \psi_k \right) - \sigma(v) Z_k \right] \\
    + \left[ 2 \alpha \mu_1 \left( \frac{v^+}{\alpha} - u_{1,k}\right) \varphi_k - 2 \mu_2 \left( \frac{v^-}{d} + u_{2,k} \right)\psi_k  \right] \\
    - \left(\mu_1 \left(1 -2 \frac{v^+}{\alpha}\right)d\psi_k + \frac{1}{d}\mu_2\left(1+2\frac{v^-}{d}\right)\alpha\varphi_k \right).
\end{multline*}
In order to conclude, we need to show that the second, third and fourth lines in the previous equation are small contributions in the $L^p_{L\textup{-per}}$ norm. Now, we just proved that the second line converges to zero in the $L^p$ topology. The third line also converges to zero, since $(u_1,u_2)_k \to \left(\frac{v^+}{\alpha}, \frac{v^-}{d}\right)$ in $\mathscr{C}^{0,\gamma}$. Finally, by \lemref{decay_eigen}, the fourth line also converges to zero in $L^p_{L\textup{-per}}$.
\end{proof}

We now recall that the solution $v$ is, by construction, linearly stable in the sense of (\ref{eq:parabolic_quasilinear_equation}). This implies in particular that any eigenpair $(\lambda,Z)$ satisfying
\begin{equation}\label{eq:eig_eq_v}
	- Z''-\left[\mu_{1}\left(1-2\frac{v^+}{\alpha}\right)\mathbf{1}_{v>0}+\frac{1}{d} \mu_{2}\left(1+2\frac{v^-}{d}\right)\mathbf{1}_{v<0}\right]Z = \lambda \sigma(v) Z
\end{equation}
is such that $\lambda$ has a positive real part. More precisely, using the uniqueness part of the Krein--Rutman theorem, we can 
establish the following convergence result.

\begin{lem}\label{lem:stab_sol}
There exists $\bar k > 0$ such that for any $k \geq \bar k$ the solution $(u_{1,k},u_{2,k})$ is linearly stable. 

Furthermore, the sequence $\left(\left(\lambda_{1,k},Z_k\right)\right)_k$ and the principal eigenpair $(\lambda_1,Z)$ given by the notion
of stability in the sense of (\ref{eq:parabolic_quasilinear_equation}) satisfy the following equalities:
\[
  \liminf_{k \to +\infty} \lambda_{1,k} = \lambda_1 > 0 \quad \textup{and} \quad \lim_{k \to +\infty} Z_k = Z 
\]
in $W^{2,p}_{L\textup{-per}}$ and $\mathscr{C}^{1,\gamma}_{L\textup{-per}}$ for any $p < \infty$ and $\gamma < 1$.
\end{lem}
\begin{proof}
In view of \lemref{conv_eign}, $(Z_k)_k$ converges to some limit 
$Z_\infty$ in $W^{2,p}_{L\textup{-per}}$ and $\mathscr{C}^{1,\gamma}$ for any $p < \infty$ and $\gamma < 1$. 
This limit is obviously an eigenfunction associated with the eigenvalue $\liminf\limits_{k \to +\infty} \lambda_{1,k}$ and, moreover, 
$Z_\infty$ is $L$-periodic, $\max Z_\infty = 1$ and $Z_\infty > 0$. Hence, by uniqueness up to normalization of the 
positive eigenfunction, the result follows. 
\end{proof}

\section*{Acknowledgments}

The authors thank Gr\'{e}goire Nadin for the attention he paid to this
work and C\'{e}cile Carr\`{e}re and Henri Berestycki for pointing out the
interest of this problem. 

\bibliographystyle{plain}
\bibliography{ref}

\end{document}